\documentclass[a4paper,12pt]{amsart}

\usepackage{amsmath, amssymb, amsthm, xcolor, enumerate}
\usepackage[colorlinks=true, citecolor=blue, linkcolor=blue, bookmarks=true]{hyperref}

\newtheorem{theorem}{Theorem}[section]
\newtheorem{lemma}[theorem]{Lemma}
\newtheorem{proposition}[theorem]{Proposition}
\newtheorem{corollary}[theorem]{Corollary}

\newtheorem{introthm}{Theorem}

\theoremstyle{definition}
\newtheorem{definition}[theorem]{Definition}

\theoremstyle{remark}
\newtheorem*{remark*}{Remark}
\newtheorem{remark}[theorem]{Remark}

\numberwithin{equation}{section}
\allowdisplaybreaks

\newcommand{\K}{\mathbb{K}}

\newcommand{\Z}{\mathcal{Z}}
\newcommand{\N}{\mathbb{N}}
\newcommand{\C}{\mathrm{C}^*}

\newcommand{\lr}[1]{\langle#1\rangle}

\title[On tracial $\Z$-stability]{On tracial $\Z$-stability of simple non-unital $\C$-algebras}
\author[J.\ Castillejos]{Jorge Castillejos}
\address{\hskip-\parindent Jorge Castillejos, Unidad Cuernavaca del Instituto de Matematicas, UNAM, Av. Universidad s/n, 62210 Cuernavaca, Morelos, México}
\email{jorge.castillejos@im.unam.mx}

\author[K.\ Li]{Kang Li}
\address{\hskip-\parindent Kang Li, Department of Mathematics, 
		Friedrich-Alexander-Universität Erlangen-N\"urnberg (FAU),
		Cauerstra\ss e 11, 91058 Erlangen, Germany}
	\email{kang.li@fau.de}

\author[G.\ Szab\'o]{G\'abor Szab\'o}
\address{\hskip-\parindent G\'abor Szab\'o, Department of Mathematics, KU Leuven, Celestijnenlaan 200b, 3001 Leuven, Belgium}
\email{gabor.szabo@kuleuven.be}

\begin{document}

\maketitle

\begin{abstract}
	We investigate the notion of tracial $\Z$-stability beyond unital $\C$-algebras, and we prove that this notion is equivalent to $\Z$-stability in the class of separable simple nuclear $\C$-algebras. 
\end{abstract}

\section*{Introduction}

The Jiang-Su algebra $\Z$ has become a cornerstone in the classification programme of simple nuclear $\C$-algebras \cite{JS99}. 
Tensorial absorption of this algebra, reminiscent of the \emph{McDuff property} for $\mathrm{II}_1$ factors, is a regularity condition which forms part of the Toms--Winter regularity conjecture (\cite{TW08}) and
it allows separable, simple, unital and nuclear $\C$-algebras in the UCT class to be classified by their K-theoretical and tracial data (\cite{TWW17,GLN20a, GLN20b, El15, Wi10, Wi12, CETWW}).
Even outside the nuclear setting, tensorial absorption of $\Z$ is a useful condition. For instance, R\o rdam showed that this condition implies almost unperforation of the Cuntz semigroup \cite{Ro04}.

For the class of simple and unital $\C$-algebras, Hirshberg and Orovitz introduced  the notion of \emph{tracial $\Z$-stability} which corresponds to a weakened form of tensorial absorption of the Jiang-Su algebra $\Z$ \cite{HO}. 
The weaker nature of this notion makes it easier to verify than tensorial absorption of  $\Z$ in many examples from dynamical systems (see for instance \cite{MR4167017, MW20, MR4066584, KN21, Niu20, Na21}).
Despite its apparent weaker form, it turns out to be equivalent to tensorial absorption of $\Z$ in the presence of nuclearity. Therefore, all such examples where one can directly verify tracial $\Z$-stability are within the scope of the classification programme.

In recent years, there has been an uptick in interest concerning the classification programme beyond unital $\C$-algebras (\cite{GL20, GL20b, GL20c, EGLN20a, EGLN20, Na20, Li21}). So it is natural to consider and study tracial $\Z$-stability for general simple $C^*$-algebras, which may not be unital. 
This notion was announced by Amini--Golestani--Jamali--Phillips and very recently appeared in \cite{AGJP}.
This paper aims to study the relation between tracial $\Z$-stability and tensorial absorption of $\Z$ in this setting. 

In particular, we show that it is preserved under stable isomorphisms and that it implies almost unperforation of the Cuntz semigroup.
Our main theorem is the following.

\begin{introthm} 
	Let $A$ be a separable simple nuclear $\C$-algebra. Then $A$ is tracially $\Z$-stable if and only if $A$ is $\Z$-stable.
\end{introthm}

It is a rather straightforward consequence of the definition that tracial $\Z$-stability yields the existence of tracially large order zero maps from matrix algebras into the central sequence algebra of $A$; a condition that we will refer to as the \emph{uniform McDuff} property. 
Not unlike in Hirshberg--Orovitz's original approach, the most significant step towards the main theorem is to show that the Cuntz semigroup of any tracially $\Z$-stable $\C$-algebra is almost unperforated.
In the presence of nuclearity this grants us access to property (SI), which allows us to make the jump to genuine $\Z$-stability with Matui--Sato's famous technique \cite{MS12} (see Section \ref{main.section}).

We hope that this version of tracial $\Z$-stability will have potential applications in the study of $\Z$-stability for algebras arising from group actions on general $\C$-algebras (see for instance \cite{AGJPII}).

\subsection*{Structure of the paper} In Section 1 we gather some preliminaries needed for this note.
In particular, we introduce a non-unital version of the uniform McDuff property and state its relation with uniform property $\Gamma$. Our version of tracial $\Z$-stability is introduced in Section 2, where we also show that it is preserved under stable isomorphisms. In Section 3, we prove that tracial $\Z$-stability yields almost unperforation of the Cuntz semigroup. Finally, the equivalence between $\Z$-stability and tracial $\Z$-stability for nuclear algebras is established in Section 4.

\subsection*{Acknowledgements}

KL and GS were supported by the Internal KU Leuven BOF project C14/19/088.
GS was further supported by the KU Leuven starting grant STG/18/019 and the project G085020N funded by the Research Foundation Flanders (FWO).
We also thank the anonymous referee for useful suggestions on an earlier version of this paper.

\begin{remark*}
The notion of tracial $\Z$-stability for non-unital $\C$-algebras is also introduced and discussed in \cite{AGJP}, which was posted online some time after this note was uploaded to the arXiv in 2021.
Some of the related results (overlapping with Sections 2 and 3) were announced by some of the authors of \cite{AGJP} at a preliminary stage at the 2018 ICM satellite conference in Rio de Janeiro (attended by the second author of this note) and COSy 2021 (attended by the first and second author of this note).
Although we claim no originality for some of the results in Sections 2 and 3, we had to ensure a self-contained treatment of our findings at the time of posting this note by supplying our own proofs (which are not necessarily identical to the ones in \cite{AGJP}) to all claims made there.
Shortly after this note was posted, the preprint \cite{FLL} appeared on arXiv where the authors developed similar ideas with the goal of finding conditions that yield stable rank one.
\end{remark*}

\section{Preliminaries}

\subsection{Notation}
	Let $A$ be a $\C$-algebra.
	We will denote the unitisation of $A$ by $\tilde{A}$, and the set of positive elements in $A$ by $A_+$. The Pedersen ideal of $A$ will be denoted by $\mathrm{Ped}(A)$. 
	We will write $\K$ for the algebra of compact operators on a separable infinite-dimensional Hilbert space and we will denote its standard matrix units by $(e_{ij})$.
	We will frequently identify $M_n(\mathbb{C})$ with the subalgebra generated by $e_{ij}$ with $1\leq i,j \leq n$. 
	
	The cone of lower semicontinuous (extended) traces on $A$ (see \cite[Section 3]{ERS11}) will be denoted by $T^+(A)$, and the set of tracial states on $A$ will be denoted by $T(A)$.
	We say $A$ is {\em traceless} when $T^+(A)$ only consists of traces with values in $\{0,\infty\}$.
	If $A$ is in addition simple, then there are precisely two (trivial) such traces denoted by the symbols $0$ and $\infty$, and the negation of $A$ being traceless is symbolically denoted as $T^+(A)\neq\{0,\infty\}$.
	Every lower semicontinuous trace on $\tau$ on $A$ extends uniquely to the lower semicontinuous trace $\tau \otimes \mathrm{Tr}$ on $A \otimes \K$ where $\mathrm{Tr}$ is the canonical unnormalised trace on $\K$ (\cite[Remark 2.27 (viii)]{BK04}). In fact, by uniqueness of $\mathrm{Tr}$ (see for instance \cite[Page 8885]{Ror19}), every lower semicontinuous trace on $A \otimes \K$ is of this form.
	\footnote{Let $\sigma$ be a non-trivial lower semicontinuous trace on $A \otimes \K$.	
	For any $a \in A_+$, the map $\sigma_a: \K_+ \to [0, \infty]$ given by $ \sigma_a(x) = \sigma(a\otimes x)$ is a well-defined lower semicontinuous trace on $\K$. 
	Similarly, $\tau: A_+ \to [0, \infty]$ given by $\tau(a) = \sigma(a \otimes e_{11})$ is a well-defined lower semicontinuous trace on $A$. Using the convention that $\infty \cdot 0 = 0$,
	the uniqueness of $\mathrm{Tr}$ yields 
	$$\sigma(a \otimes x) = \sigma_a(x) = \sigma_a(e_{11})\mathrm{Tr}(x) = \sigma(a \otimes e_{11}) \mathrm{Tr}(x) = \tau(a) \mathrm{Tr}(x).$$
	}
	We will use this fact freely.
	We will also simply say trace instead of lower semicontinuous trace.
		
	Given a $\C$-algebra $B$ and a map $\varphi: A \to B$. If $\mathfrak F \subseteq B$ is a subset, we write $\|[\varphi, \mathfrak{F}]\|< \epsilon$ to mean that $\|[\varphi(a),b]\|< \epsilon$ for any contractions $a \in A$ and $b \in \mathfrak{F}$.
	We will use $a \approx_\epsilon b$ to denote $\|a-b\|<\epsilon$. If $a\in A_+$ and $\epsilon>0$, then $(a-\epsilon)_+$ denotes the positive part of the self-adjoint element $a-\epsilon 1_{\tilde{A}}$. 
	
	As usual, given a free ultrafilter $\omega$ on $\N$, we denote the \emph{$\C$-ultrapower} of $A$ by $A_\omega:=\ell^\infty(A)/\{(a_n)_{n \in \mathbb{N} }|\lim_{n\to \omega}\|a_n\|=0\}$.
	For any subalgebra $D\subset A_\omega$, we denote $A_\omega \cap D'=\{x\in A_\omega \mid [x,D]=\{0\} \}$ and $A_\omega \cap D^\perp = \{x \in A_\omega \mid xD = Dx = \{0\}\}$.
	The \emph{(corrected) central sequence algebra} (see \cite[Definitions 1.1]{Ki06}) $F_\omega(A)$ of a $\sigma$-unital $\C$-algebra $A$ is given by the quotient
	\begin{align}
		F_\omega(A):= (A_\omega \cap A') / (A_\omega \cap A^\perp).
	\end{align}
	
\subsection{Functional calculus of order zero maps}\label{section:orderzero}

A c.p.\ map $\varphi: A \to B$ between $\C$-algebras has \emph{order zero} if it preserves orthogonality, i.e.\ $\varphi(a)\varphi(b)=0$ if $a,b\in A_+$ satisfy $ab=0$. By the structure theorem proved by Winter and Zacharias \cite[Theorem~3.3]{WZ09}, $\varphi: A \to B$ has order zero if and only if there are $h\in  \mathcal M(\C(\varphi(A)))_+$ and a $^*$-homomorphism $\pi: A \to \mathcal M(\C(\varphi(A)))\cap \{h\}'$ such that $\varphi(a) = h \pi(a)$ for $a\in A$. If $A$ is unital, $h$ is equal to $\varphi(1)$.
Using this decomposition, for any positive function $f \in C_0(0,1]$, one can define a new c.p.\ order zero map $f(\varphi): A \to B$ by $f(\varphi)(a)=f(h)\pi(a)$ for $a\in A$ (see \cite[Corollary~4.2]{WZ09}).

	\subsection{Cuntz equivalence and strict comparison}
	
	Given $a,b \in A_+$, it is said that $a$ is \emph{Cuntz-below} $b$, denoted $a \precsim_A b$, if for any $\epsilon>0$ there is $x\in A$ such that $x^*bx \approx_\epsilon a$. We will often simply write $\precsim$ if the relevant $\C$-algebra is understood from the context. It is said that $a$ is \emph{Cuntz-equivalent} to $b$, denoted by $a\sim b$,  if $a\precsim b$ and $b \precsim a$.
	The \emph{Cuntz semigroup} is given by $\mathrm{Cu}(A) = (A \otimes \K)_+ / \sim$ equipped with orthogonal addition
	and order given by Cuntz-subequivalence. The equivalence class of $a\in (A \otimes \K)_+$ in $\mathrm{Cu}(A)$ is denoted by $\langle a \rangle$. We refer to reader to \cite[Section 2]{APT11} for a more comprehensive review of the basic construction of the Cuntz semigroup.

	The Cuntz semigroup $\mathrm{Cu}(A)$ is called \emph{almost unperforated} if $\langle a \rangle \leq \langle b \rangle$ holds whenever $\langle a \rangle, \langle b \rangle \in \mathrm{Cu}(A)$ satisfy $(k+1)\langle a \rangle  \leq k \langle b \rangle$ for some $k\in \N$.	
	For simplicity, we will simply say quasitrace instead of lower semicontinuous $2$-quasitrace (see \cite[Definition 2.22]{BK04}).
	Any quasitrace on $A$ extends uniquely to the quasitrace $\overline{\tau}:=\tau \otimes \mathrm{Tr}$ on $A\otimes \K$ by \cite[Remark 2.27]{BK04}.
	The dimension function $d_\tau: (A\otimes \K)_+ \to [0,\infty]$ associated to $\tau$ is given by $d_\tau (a) = \lim_{n \to \infty} \overline{\tau}(a^{1/n})$. This dimension function induces a well-defined functional on $\mathrm{Cu}(A)$. 

	It is said that $A$ has \emph{strict comparison} if $\langle a \rangle \leq \langle b\rangle$ holds whenever $d_\tau(\langle a \rangle )< d_\tau(\langle b \rangle )$ for every quasitrace $\tau$ on $A$
	with $d_\tau (\langle b \rangle)>0$.
	If $A$ is simple  then  $\mathrm{Cu}(A)$ is almost unperforated if and only if $A$ has strict comparison (see \cite[Proposition 5.2.20]{APT} or \cite[Remark 9.2 (3)]{Thi20}).  
	
	We finish this subsection with the following definition. 
	
	\begin{definition}[{\cite[Definition 5.3.1]{APT}}]\label{def:soft}
	Let $x,y \in \mathrm{Cu}(A)$. It is said that $x$ is \emph{way-below} $y$, denoted $x \ll y$, if whenever $( y_n )_{n\in \mathbb{N}}$ is an increasing sequence with $y \leq  \sup_{n \in \mathbb{N}} y_n$, then
	there is some $n \in \mathbb{N}$ such that $x \leq y_n$. An element $y \in \mathrm{Cu}(A)$ is \emph{soft} if for any $x \in \mathrm{Cu}(B)$ such that $x \ll y$ there is some $k:=k(x) \in  \mathbb{N}$ such that $(k+1) x \leq ky$.
	\end{definition}

\subsection{Generalised limit traces}
	
	A trace $\tau$ on $A_\omega$ is called a \emph{limit trace} if there is a sequence of tracial states $(\tau_n)_{n \in \mathbb{N} }$ on $A$ such that $\tau((a_n)_{n\in \mathbb{N} }) = \lim_{n \to \omega} \tau_n(a_n)$.
	Such tracial states on $A_\omega$ capture enough about the global tracial information when $A$ is unital, but become too restrictive when $A$ has unbounded traces. 
	In the general non-unital case, it becomes more appropriate to study the following class of traces.

\begin{definition}[{\cite[Definition 2.1]{Szabo19}}]
	Let $\omega$ be a free ultrafilter on $\mathbb{N}$. For a sequence of lower semicontinuous traces $(\tau_n)$ on $A$, one defines a lower semicontinuous trace $\tau: \ell^\infty (A)_+ \to [0, \infty]$ by
	\begin{equation}
	\tau((a_n)_{n \in \mathbb{N} }) = \sup_{\epsilon>0} \lim_{n \to \omega} \tau_n ((a_n - \epsilon)_+), \qquad (a_n)_{n \in \mathbb{N} } \in \ell^\infty(A)_+.
	\end{equation}
	This trace induces a lower  semicontinuous trace on $A_\omega$. Traces of this form on $A_\omega$ are called \emph{generalised limit traces}.
	The set of all generalised limit traces on $A_\omega$ will be denoted by $T_\omega^+(A)$. 
	
\end{definition}

Given a generalised limit trace $\tau$ on $A_\omega$ and $a\in A_+$, the map $\tau_a: (A_\omega \cap A')_+ \to [0, \infty], x \mapsto \tau(ax)$ yields a trace that satisfies $\tau_a (x) \leq \tau(a)\|x\|$ for $x \in (A_\omega \cap A')_+$.
It actually extends to a positive tracial functional on $A_\omega \cap A'$ with norm $\|\tau_a\|=\tau(a)$ if $\tau(a)<\infty$.
Moreover, this induces a trace on $F_\omega(A)$, which we will also denote by $\tau_a$.
We refer the reader to \cite[Remark~2.3]{Szabo19} for details.

\begin{lemma}[{cf.\ \cite[Proposition 2.3]{CEII}}]\label{lem:generalised.vs.limit}
	Let $A$ be an algebraically simple $\sigma$-unital $\C$-algebra with $T^+(A) \neq \{0,\infty\}$. 
	If a generalised limit trace on $A_\omega$ is finite on some non-zero positive element of $A$, then it is a constant multiple of a limit trace.
\end{lemma}

\begin{proof}
	Since $A$ is algebraically simple, all non-trivial lower semicontinuous traces on $A$ are bounded and extend to positive tracial functionals.
	The same proof of \cite[Proposition 2.3]{CE} remains valid except we replace compactness of $T(A)$ with \cite[Proposition 2.5]{Ti14} in order to obtain $\inf_{\sigma \in T(A)} \sigma(a) > 0$ for some non-zero $a \in A_+$. 
\end{proof}

\subsection{Uniform McDuff property for non-unital $\C$-algebras}

We now present a non-unital version of the uniform McDuff property, which was defined in \cite[Definition 4.2]{CETW} for $\C$-algebras with compact tracial state space. 
As explained in \cite[Remark 4.3]{CETW}, the uniform McDuff property is equivalent to the existence of a c.p.c.\ order zero map $\varphi: M_n(\mathbb{C}) \to A_\omega \cap A'$ such that $\tau ( \varphi(1_{M_n})) =1$ for any limit trace $\tau$ on $A_\omega$; such maps have been called (uniformly) \emph{tracially large}. We also state here a version of such maps in the general non-unital case.

\begin{definition}\label{def:traciallylargemap}
	Let $A$ and $B$ be $\C$-algebras with $B$ unital and $A$ separable simple and $T^+(A)\neq \{0,\infty\}$. 
	A c.p.c.\ order zero map $\varphi: B \to F_\omega(A)$ is \emph{tracially large} if $\tau_a (\varphi(1)) = \tau(a)$ for all (or any) non-zero positive element $a \in \mathrm{Ped}(A)$ and $\tau \in T^+_\omega(A)$ with $\tau(a) < \infty$.\footnote{In other words, the element $1-\varphi(1)$ vanishes under the trace $\tau_a$. The fact that ``for all'' and ``for any'' agree here is due to \cite[Proposition 2.4]{Szabo19} applied to the inclusion $A\subset A_\omega$.}
	
	We say that a separable, simple $\C$-algebra $A$ with $T^+(A)\neq\{0,\infty\}$ is \emph{uniformly McDuff} if for every $n\in \mathbb{N}$ there is a tracially large order zero map $\varphi: M_n(\mathbb{C}) \to F_\omega(A)$.
\end{definition}

As in \cite[Proposition 2.6.(ii)]{CEII}, the notion introduced here agrees with Definition 4.2 of \cite{CETW} if every $\tau \in T^+(A)\setminus\{\infty\}$ is bounded and the tracial state space is compact and non-empty.
One may hence view the above as some kind of stabilised version of the previously defined notion of uniform McDuffness.

Since stably isomorphic $\C$-algebras have isomorphic central sequence algebras, it is reasonable to expect that the uniform McDuff property is also preserved. In order to prove this, one has to check how the stable isomorphism is transforming the traces of the form $\tau_a$ with $a \in \mathrm{Ped}(A)$ and $\tau \in T_\omega^+(A)$.

\begin{proposition} \label{Mcduff stable iso}
	The uniform McDuff property is preserved under stable isomorphism.
\end{proposition}

\begin{proof}
By \cite[Theorem 1.4]{Ped66}, any isomorphism between two $\C$-algebras restricts to an isomorphism between their Pedersen ideals.
If we combine this with \cite[Proposition 2.8]{CEII}, it follows directly that the uniform McDuff property is preserved under genuine isomorphism.
Therefore, given a separable simple $\C$-algebra $A$ with $T^+(A)\neq \{0,\infty\}$, it is enough to check that $A$ is uniformly McDuff if and only if $A \otimes \K$ is uniformly McDuff.
	
Consider $m\in \N$ and suppose $\varphi: M_m(\mathbb{C}) \to F_\omega (A)$ is any completely positive contractive order zero map, which by \cite[Proposition 1.2.4]{Wi09} can be represented by a sequence of order zero maps $\varphi_n: M_m(\mathbb{C})\to A$.
Let $\Psi: A_\omega \to (A\otimes \K)_\omega$ be given by $(a_n)_{n \in \mathbb{N}} \mapsto (a_n \otimes (\sum_{i=1}^n e_{ii}) )_{n \in \mathbb{N}}$.
By \cite[Lemma 1.3]{CEII} (which goes back to \cite{Ki06}), $\Psi$ induces an isomorphism $\bar{\Psi}: F_\omega(A) \to F_\omega (A \otimes \K)$.	
Consider a generalised limit trace $\sigma \in T^+_\omega(A\otimes \mathbb{K})$.
This trace is induced by a sequence $(\tau_n \otimes \mathrm{Tr})_{n \in \mathbb{N}}$ where $\tau_n \in T^+(A)$.
Let $\tau$ be the generalised limit trace on $A_\omega$ induced by the sequence $(\tau_n)_{n \in \mathbb{N}}$.
Suppose $a\in\mathrm{Ped}(A)\setminus\{0\}$ is positive.
Then we observe
\[
\begin{array}{cl}
\multicolumn{2}{l}{ \sigma_{a\otimes e_{11}}(\bar{\Psi}(\varphi(1))) }\\
=& \displaystyle \sup_{\varepsilon>0} \lim_{n\to\omega} (\tau_n\otimes\mathrm{Tr})\Big( (a^{1/2}\otimes e_{11})\cdot \varphi_n(1)\otimes\big(\sum_{j\leq n} e_{jj}\big)\cdot (a^{1/2}\otimes e_{11}) -\epsilon\Big)_+ \\
=& \displaystyle \sup_{\varepsilon>0} \lim_{n\to\omega} \tau_n\big( (a^{1/2}\varphi_n(1)a^{1/2} - \epsilon)_+ \big)  \ = \ \tau_a(\varphi(1)).
\end{array}
\]
Now $\varphi$ is tracially large if and only if $\tau_a(\phi(1))=\tau(a)$ whenever $0<\tau(a)<\infty$, and $\bar{\Psi}\circ\varphi$ is tracially large if and only if $\sigma_{a\otimes e_{11}}(\bar{\Psi}(\varphi(1)))=\sigma(a\otimes e_{11})$ whenever $0<\sigma(a\otimes e_{11})<\infty$.
Since by construction we always have $\tau(a)=\sigma(a\otimes e_{11})$ for the aforementioned assignment $\sigma\mapsto\tau$, the computation above directly implies that $\varphi$ is tracially large if and only if $\bar{\Psi}\circ\varphi$ is tracially large.
Consequently, $A$ is uniformly McDuff if and only if $A \otimes \K$ is.
\end{proof}

\begin{remark}	 \label{rem:Z-is-McDuff}
If $A$ is a separable simple exact $\Z$-stable $\C$-algebra with $T^+(A)\neq\{0,\infty\}$, then it is uniformly McDuff.
Indeed, let us first assume that every non-trivial trace on $A$ is a multiple of a tracial state and that $T(A)$ is compact.
Let $n \in \mathbb{N}$.
By \cite[Proposition 2.3]{CETWW}, there is an order zero map $\varphi: M_n(\mathbb{C}) \to F_\omega(A)= A_\omega \cap A'$ such that $\tau (\varphi(1)) = 1$ for all $\tau \in T_\omega (A)$.
It follows that $A$ is uniformly McDuff. 
We hence observe that the general case follows with the previous proposition and \cite[Theorem 2.7]{CE}. 
\end{remark}

In light of \cite[Theorem 4.6]{CETW}, it is natural to ask about the relation between the uniform McDuff property and stabilised uniform property $\Gamma$ (see \cite[Definition~2.5]{CEII}).
We can at least provide an answer to this question if we assume a regularity condition at the level of the Cuntz semigroup.
(It is conceivable that, although our proof uses these assumptions, the statement may actually hold in greater generality.)

\begin{proposition}
	Let $A$ be a simple separable $\C$-algebra with $T^+(A)\neq \{0,\infty\}$.
	If $A$ is uniformly McDuff, then $A$ has stabilised uniform property $\Gamma$. 
	The converse holds if $A$ is nuclear and has either stable rank one or $\mathrm{Cu}(A \otimes \Z) \cong \mathrm{Cu}(A)$.
\end{proposition}

\begin{proof}
	Suppose $A$ is uniformly McDuff.
	Then for any $n\in \N$ there is a tracially large order zero map $\varphi: M_n(\mathbb{C}) \to F_\omega(A)$.
	It follows that the map $x \mapsto \tau_a(\varphi(x))$ is a tracial functional on $M_n(\mathbb{C})$ for any $a \in A_+$ and $\tau \in T^+_\omega(A)$ with $\tau(a)<\infty$. 
	If $\mathrm{tr}$ denotes the normalised trace on $M_n(\mathbb{C})$, then
\begin{align}
\tau_a (\varphi (x)) = \tau_a(\varphi(1))\mathrm{tr}(x)= \tau(a) \mathrm{tr}(x), \qquad x\in M_n(\mathbb{C}).
\end{align}	
	
Observe that the last equality follows from the fact that $\varphi$ is tracially large.
If we set $f_i = \varphi(e_{ii})$, then the previous equation shows $\tau_a (f_i) = \frac{1}{n}\tau(a)$  for any $a \in A_+$ and $\tau \in T^+_\omega(A)$ with $\tau(a)<\infty$. 
Hence $A$ has stabilised uniform property $\Gamma$.

For the converse implication, we suppose that $A$ is nuclear and has either stable rank one or satisfies $\mathrm{Cu}(A \otimes \Z) \cong \mathrm{Cu}(A)$. 
By \cite[Proposition~2.4]{CE}, $A$ is either stably projectionless or stably isomorphic to a unital $\C$-algebra.
In the stably projectionless case, $A$ is also stably isomorphic to a $\C$-algebra $B$ with non-empty compact tracial state space and without unbounded traces  by \cite[Lemma 3.3]{CEII} and the final observation in \cite{CEII} (this uses the assumption that either $\mathrm{Cu}(A) = \mathrm{Cu}(A \otimes \Z)$ or $A$ has stable rank one, respectively).

Since $A$ is non-elementary, simple, separable and nuclear, $B$ is also non-elementary, simple, separable and nuclear. 
By \cite[Theorem 4.6]{CETW}, $B$ is uniformly McDuff if and only if $B$ has uniform property $\Gamma$.
Then, since stabilised property $\Gamma$ and the McDuff property are preserved under stable isomorphisms (see \cite[Theorem~2.10 and Proposition 2.6]{CEII} and Proposition~\ref{Mcduff stable iso}), we deduce that $A$ is uniformly McDuff if and only if $A$ has stabilised uniform property $\Gamma$.
\end{proof}

\section{Tracial $\Z$-stability}

Let us begin by stating the general version of tracial $\Z$-stability. 
This notion uses the idea of `tracial smallness' that can be traced back to \cite{Li01, Li01-2, Ph11, HO, FG}.  

\begin{definition}[see also {\cite[Definition 3.6]{AGJP}}] \label{non-unital tzs}
	A simple $\C$-algebra $A$ is tracially $\Z$-stable if $A \neq \mathbb{C}$ and for any finite set $\mathfrak{F}\subseteq A$, $\epsilon>0$, non-zero positive elements $a, b \in A_+$, and $n\in \mathbb{N}$ there is a c.p.c.\ order zero map $\varphi: M_n(\mathbb{C}) \to A$ such that
\[
(b(1_{\tilde{A}} - \varphi(1))b-\epsilon)_+ \precsim  a
\quad\text{and}\quad
\| [\varphi, \mathfrak{F}] \| < \epsilon.
\]
\end{definition}

\begin{remark}
	When $A$ is a unital $\C$-algebra, it is straightforward to see that \cite[Definition 2.1]{HO} implies Definition~\ref{non-unital tzs}.
	The converse implication follows from functional calculus of c.p.c.\ order zero maps. More precisely, we consider the map $f(\varphi)$ where $f$ is a continuous function on $[0,1]$ given by
	\begin{align*}
	f(t):=\begin{cases} 
	(1-\epsilon)^{-1}t & \quad t \in [0,1-\epsilon], \\
	1 & \quad t \in [1-\epsilon,1].
	\end{cases}
	\end{align*} 
	This order zero map satisfies that $1-f(\varphi)(1) \sim (1-\varphi(1)-\epsilon)_+$ (see the proof of \cite[Lemma 2.8]{ABP}). 
\end{remark}

\begin{remark}
	\begin{enumerate}[(a)]
	\item We also note that in principle we allow $\varphi$ to be the zero map in order to consider purely infinite simple $\C$-algebras as tracially $\Z$-stable (see also \cite[Proposition~3.5(ii)]{KR00}). 
	This has the one downside that it would a priori also include $\mathbb{C}$ as an example, which is why we exclude this specific case in the definition.

	\item We can furthermore observe that simple $\C$-algebras cannot be tracially $\Z$-stable if they are elementary, i.e., of type I.
	Indeed, suppose $A=\mathbb K(\mathcal H)$ for a Hilbert space of dimension at least 2.
	Let $p_1, p_2\in A$ be two orthogonal rank one projections, and $v\in A$ a partial isometry with $v^*v=p_1$ and $vv^*=p_2$.
	If we consider $a=p_1$, $b=p_1+p_2$ and $\mathfrak F=\{v,v^*\}$ and choose a c.p.c.\ order zero map $\varphi$ for $n=2$ as in the definition of tracial $\Z$-stability for small enough $\epsilon$, then we obtain a contradiction.
	This is because if $x\in M_2(\mathbb C)$ is any contraction, then the operator $\varphi(x)$ not only leaves the 2-dimensional subspace $(p_1+p_2)\mathcal H$ approximately invariant (as $[\varphi(x),b]\approx 0$), but $\|[\varphi(x),\mathfrak F]\|\approx 0$ forces $\varphi(x)$ to act like a constant multiple of the identity on this subspace. 
	Considering the domain of $\varphi$, it can hence only be order zero if $b\varphi(\cdot)b$ is close to the zero map.
	For sufficiently small $\epsilon$ this would yield $\frac12(p_1+p_2)\leq (b(1_{\tilde{A}} - \varphi(1))b-\epsilon)_+ \precsim a=p_1$, which is not possible.
		\end{enumerate}
\end{remark}

The following is rather straightforward.

\begin{lemma}\label{new char}
Let $A$ be a simple $\C$-algebra with $A\neq\mathbb C$.
Let $\mathcal S\subseteq A_+$ be any subset containing an approximate unit of contractions.
Then $A$ is tracially $\Z$-stable if and only if for any finite set $\mathfrak{F}\subseteq A$, $\epsilon>0$, non-zero positive contraction $a \in A$, $b\in\mathcal S$, and $n\in \mathbb{N}$ there is a c.p.c.\ order zero map $\varphi: M_n(\mathbb{C}) \to A$ such that
\[
(b(1_{\tilde{A}} - \varphi(1))b-\epsilon)_+ \precsim  a
\quad\text{and}\quad
\| [\varphi, \mathfrak{F}] \| < \epsilon.
\]
\end{lemma}
\begin{proof}
The ``only if'' part is tautological.
For the ``if'' part, let us fix an arbitrary tuple $(\mathfrak{F},\epsilon,a,b,n)$, in particular with some arbitrary element $b\in A_+$ with norm one.
By the property of $\mathcal S$, we may choose a positive contraction $h\in\mathcal S$ with $\|b-bh\|<\epsilon/4$.
By assumption, we can find a c.p.c.\ order zero map $\varphi: M_n(\mathbb{C}) \to A$ such that
\[
\big( h(1_{\tilde{A}} - \varphi(1))h-\frac{\epsilon}{2} \big)_+ \precsim  a
\quad\text{and}\quad
\| [\varphi, \mathfrak{F}] \| < \epsilon.
\]
Using the well-known \cite[Lemma 2.2]{KR02}, we observe
\[
\renewcommand*{\arraystretch}{1.2}
\begin{array}{ccl}
(b(1_{\tilde{A}} - \varphi(1))b-\epsilon)_+ &\precsim& (bh(1_{\tilde{A}} - \varphi(1))hb-\frac{\epsilon}{2})_+ \\
&\precsim& b(h(1_{\tilde{A}} - \varphi(1))h-\frac{\epsilon}{2})_+ b \\
&\precsim& (h(1_{\tilde{A}} - \varphi(1))h-\frac{\epsilon}{2})_+ \ \precsim \ a.
\end{array} 
\]  
\end{proof}

\begin{proposition} \label{prop:tZ.ultrapower}
	Let $A$ be a simple $\C$-algebra with $A\neq\mathbb C$.
	Then $A$ is tracially $\Z$-stable if and only if for any $n \in \mathbb{N}$, any separable $\C$-subalgebra $D\subseteq A$, and any positive element $a \in A_\omega$ of norm one, there exist a c.p.c.\ order zero map $\varphi:M_n(\mathbb{C}) \to A_\omega \cap D'$ and a contraction $x\in A_\omega$ such that 
	\begin{equation}
	ax=x \quad\text{and}\quad x^*x-(1_{\tilde{A}_\omega}-\varphi(1)) \in \tilde{A}_\omega\cap D^\perp. 
	\end{equation}
\end{proposition}

\begin{proof}
Suppose that $A$ is tracially $\Z$-stable.
Let $(a_k)_{k\in \N}$ be a sequence of positive elements of norm one representing $a\in A_\omega$ and fix $n \in \mathbb{N}$.
By employing functional calculus, we may perturb $(a_n)_{n\in \mathbb{N}}$ by a null sequence and assume without loss of generality that there exists another sequence of norm one positive elements $(d_k)_{k\in\N}$ with $d_ka_k=d_k$ for all $k\in\N$.
We choose a countable increasing approximate unit $\{e_k\}_{k\in \N}$ in $D_+$, and let $e\in D_\omega\subset A_\omega$ be its induced element.
Let $\mathfrak{F}_k \subset D$ be an increasing sequence of finite subsets with dense union.

Using the hypothesis, we can find a sequence of c.p.c.\ order zero maps $\varphi_k: M_n(\mathbb{C}) \to A$ such that
	\begin{align}
	\Big( e_k (1_{\tilde{A}}-\varphi_k(1)) e_k - \frac{1}{k}\Big)_+ \precsim d_k \quad \text{and} \quad \| [\varphi_k, \mathfrak{F}_k] \| < \epsilon. 
	\end{align}
In particular we may find a (possibly unbounded) sequence $r_k\in A$ satisfying
	\begin{align}
	r_k^*d_kr_k \approx_{2/k} e_k( 1_{\tilde{A}} - \varphi_k (1)) e_k,\quad k\in\N.
	\end{align}

	We see that the sequence $x_k=\sqrt{d_k}r_k$ satisfies $\limsup_{k\to\infty} \|x_k\|\leq 1$, so we obtain a contraction $x\in A_\omega$ induced by this sequence.
	Let $\varphi: M_n(\mathbb{C}) \to A_\omega$ be the sequence induced by $(\varphi_k)$. 
Clearly $\varphi$ is a c.p.c.\ order zero map and by construction the image of $\varphi$ is actually in the relative commutant $A_\omega \cap D'$. 
Furthermore we have arranged that
\[ 
ax=x \quad\text{and}\quad x^*x=e( 1_{\tilde{A}} - \varphi (1))e.
\]
Since $e$ acts like a unit on elements of $D$, we see that this implies the required condition $x^*x-( 1_{\tilde{A}_\omega} - \varphi (1))\in \tilde{A}_\omega\cap D^\perp$.

	Conversely, let $\mathfrak{F} \subseteq A$ be a finite subset, $\epsilon> 0$ and consider non-zero positive elements $a, b \in A$, and any $n\in \mathbb{N}$.
	Let us assume without loss of generality $\|a\|=1$ and $\mathfrak F = \mathfrak F^*$.
	By the hypothesis, there is an order zero map $\varphi: M_n(\mathbb{C}) \to A_\omega \cap \mathfrak F'$ and a contraction $x\in A_\omega$ such that $ax=x$ and $x^*x-(1_{\tilde{A}_\omega}-\varphi(1))\in \tilde{A}_\omega\cap\{b\}^\perp$. 
	A particular consequence of this is $bx^*axb=b(1_{\tilde{A}_\omega}-\varphi(1))b$.
	By \cite[Proposition 1.2.4]{Wi09}, we can find a sequence of order zero maps $\varphi_k: M_n(\mathbb{C}) \to A$ that induces $\varphi$.
	Likewise, choose a sequence of contractions $x_k\in A$ representing $x$.
	Then this leads to the limit behavior
	\begin{align}
	\lim_{k \to \omega} \| bx_k^* a x_kb - b(1_{\tilde{A}}-\varphi_k(1))b \| = 0. 
	\end{align}
	Thus, there is $I\in \omega$ such that for all $k \in I$ one has
	\begin{align}
	\|bx_k^* a x_kb - b(1_{\tilde{A}}-\varphi_k(1))b \| < \epsilon.
	\end{align}
It follows that $(b(1_{\tilde{A}}-\varphi_k (1))b - \epsilon)_+ \precsim bx_k^* a x_kb \precsim a$. 
	Since the image of $\varphi$ is in the relative commutant $A_\omega \cap \mathfrak F'$, we can also assume that $\|[\varphi_k, \mathfrak F]\|<\epsilon$ for suitably chosen $k$.
	This shows that $A$ is tracially $\Z$-stable.
\end{proof}

We will now prove some permanence properties for tracial $\Z$-stability.

\begin{proposition}[see also {\cite[Theorem 4.1]{AGJP}}] \label{prop:hereditary}
	Let $A$ be a simple $\C$-algebra and suppose $B\subseteq A$ is a hereditary $\C$-subalgebra.
	If $A$ is tracially $\Z$-stable, then so is $B$.
\end{proposition}
\begin{proof}
Since $A$ is tracially $\Z$-stable, it cannot be of type I.
Since $A$ is also simple and $B$ is hereditary, we may conclude $B\neq\mathbb C$.
So it suffices to show that the condition in Proposition \ref{prop:tZ.ultrapower} passes from $A$ to $B$.
Let $D\subseteq B$ be a separable $\C$-subalgebra and $b\in B_\omega$ a positive element of norm one.
A standard application of the $\epsilon$-test yields a positive norm one element $e\in B_\omega$ such that $ed=d=de$ for all $d\in D\cup\{b\}$.
Let $n\in\N$.
Using that $A$ is tracially $\Z$-stable, we find a c.p.c.\ order zero map $\psi: M_n(\mathbb C)\to A_\omega\cap D'\cap\{e\}'$ and a contraction $x\in A_\omega$ such that $bx=x$ and
\[
1_{\tilde{A}_\omega} +\tilde{A}_\omega\cap D^\perp =x^*x+\psi(1) + \tilde{A}_\omega\cap D^\perp
\] 
in $(\tilde{A}_\omega\cap D')/(\tilde{A}_\omega\cap D^\perp)$.
By these properties of $\psi$ and the fact that $B$ is hereditary, we see that $\varphi=e\psi(\cdot)e: M_n(\mathbb C)\to B_\omega\cap D'$ is also c.p.c.\ order zero.
Furthermore we obtain the equality 
\[
1_{\tilde{A}_\omega} + \tilde{A}_\omega\cap D^\perp= x^*x+\psi(1) + \tilde{A}_\omega\cap D^\perp=ex^*xe+\varphi(1) + \tilde{A}_\omega\cap D^\perp
\]
in $(\tilde{A}_\omega \cap D' )/ (\tilde{A}_\omega \cap D^\perp)$.
Since clearly $bxe=xe$, we have that $xe\in B_\omega$ is a contraction, and the equation from left to right actually holds in $(\tilde{B}_\omega \cap D')/(\tilde{B}_\omega\cap D^\perp)$.
This finishes the proof.
\end{proof}

We now prove that tracial $\Z$-stability passes to minimal tensor products with arbitrary simple $\C$-algebras.
This observation in particular generalizes \cite[Lemma 2.4]{HO}.

\begin{proposition}[see also {\cite[Theorem 5.1]{AGJP}}] \label{prop:stabilisation}
	Let $A$ and $B$ be two simple $\C$-algebras.
	If $A$ is tracially $\Z$-stable, then so is the minimal tensor product $A \otimes B$.
\end{proposition}
\begin{proof}
Let $\mathfrak F_A\subset A$ and $\mathfrak F_B\subset B$ be finite sets of contractions.
It suffices to check the condition in Definition \ref{non-unital tzs} for sets of the form $\mathfrak F=\{ a\otimes b\mid a\in\mathfrak F_A,\ b\in \mathfrak F_B\}$.
Let $g\in A\otimes B$ be a non-zero positive element, which will play the role of the element $a$ in Definition \ref{non-unital tzs}.
By Kirchberg's Slice Lemma \cite[Lemma 4.1.9]{Rordam}, we can find a pair of non-zero positive elements $a_1\in A$ and $b_0\in B$ with $a_1\otimes b_0\precsim g$.
Appealing to Lemma \ref{new char}, it suffices to check the condition in Definition \ref{non-unital tzs} for elementary tensors of norm one positive contractions in place of arbitrary elements $b$ as stated there.

Let $0<\epsilon<1$ and $n\in\N$ be given.
Let $e\in A$ and $f\in B$ be any pair of positive contractions of norm one. 
Since $B$ is simple, we can find some natural number $k\in\N$ with 
\begin{equation} \label{eq:tensors:1}
\langle (f^2 -\epsilon)_+ \rangle \leq k\langle b_0\rangle.
\end{equation}
Appealing to \cite[Lemma 2.3]{HO}\footnote{We note that despite unitality being an assumption there, all that is needed in the proof is that hereditary subalgebras of $A$ are not type I, in order to appeal to Glimm's theorem via \cite[Proposition 4.10]{KR00}. This is automatic here.},
we find a non-zero positive contraction $a_0\in A$ with $k\langle a_0\rangle \leq \langle a_1\rangle$.

Now use that $A$ is tracially $\Z$-stable and choose a c.p.c.\ order zero map $\psi: M_n(\mathbb C)\to A$ satisfying
\begin{equation} \label{eq:tensors:2}
(e(1_{\tilde{A}}-\psi(1))e-\epsilon)_+\precsim a_0 \quad\text{and}\quad \|[\psi,\mathfrak F_A]\|<\epsilon.
\end{equation}
Let $s\in B$ be a positive contraction that satisfies $sx \approx_\epsilon x \approx_\epsilon xs$ for all $x \in \mathfrak{F}_B \cup \{f\}$.
Then $\varphi=\psi\otimes s: M_n(\mathbb C)\to A\otimes B$ is another c.p.c.\ order zero map that clearly satisfies $\|[\varphi,\mathfrak F]\|<3\epsilon$.
Then
\[
\begin{array}{cll}
(e\otimes f)(1-\varphi(1))(e\otimes f) &\approx_\epsilon& (e\otimes f)(1_{\tilde{A}}\otimes s-\varphi(1))(e\otimes f) \\
&=& e(1-\psi(1))e\otimes fsf \\
&\approx_\epsilon& e(1-\psi(1))e\otimes f^2
\end{array}
\]
and hence we observe (appealing again to \cite[Lemma 2.2]{KR02}) that
\[
\renewcommand\arraystretch{1.3}
\begin{array}{cl}
\multicolumn{2}{l}{ \big\langle \big( (e\otimes f)(1-\varphi(1))(e\otimes f) - 4\epsilon \big)_+ \big\rangle } \\
\leq & \big\langle \big( e(1-\psi(1))e\otimes f^2 - 2\epsilon \big)_+ \big\rangle \\
\leq& \big\langle \big( e(1-\psi(1))e - \epsilon\big)_+ \otimes (f^2-\epsilon)_+ \big\rangle \\
\stackrel{\eqref{eq:tensors:1},\eqref{eq:tensors:2}}{\leq}& k\langle a_0 \otimes b_0\rangle \ \leq \ \langle a_1\otimes b_0\rangle \ \leq \ \langle g\rangle.
\end{array}
\]
This verifies the condition in Lemma \ref{new char} for $A\otimes B$.
\end{proof}

A straightforward consequence of Propositions \ref{prop:hereditary} and \ref{prop:stabilisation} is: 

\begin{corollary}[see also {\cite[Proposition 4.11]{AGJP}}]\label{stable iso}
	Let $A$ be a simple $\C$-algebra.
	Then $A$ is tracially $\Z$-stable if and only if $A\otimes \K$ is tracially $\Z$-stable.
\end{corollary}

\section{Strict comparison}

We move now to show that simple tracially $\Z$-stable $\C$-algebras have almost unperforated Cuntz semigroups.
For this we will mimic the original approach used by Hirshberg--Orovitz in \cite[Lemma 3.2]{HO}. The results appearing in this section were announced some years ago by the authors of \cite{AGJP}. 

\begin{lemma}[{cf.\ \cite[Lemma 3.2]{HO}}]\label{lem:strict.comp}
	Let $A$ be a simple tracially $\mathcal{Z}$-stable $\C$-algebra. 
	Let $a,b\in A_+$ and suppose that $0$ is an accumulation point of $\sigma(b)$.
	If $k \lr a \leq k \lr b$ in ${\rm Cu}(A)$ for some $k\in \N$, then $a\precsim b$.
\end{lemma}
\begin{proof}
	We will proceed as the original proof of \cite[Lemma 3.2]{HO} with some modifications. 
	Since we have shown that hereditary subalgebras of $A$ are tracially $\Z$-stable, we may assume without loss of generality that $A$ is $\sigma$-unital.
	Let us fix $\epsilon>0$. Without loss of generality, we may assume that $a$ and $b$ have norm equal to $1$. Let $c=(c_{ij}) \in M_k(A)$ and $\delta>0$ such that $c((b-\delta)_+ \otimes 1_k) c^* = (a-\epsilon)_+ \otimes 1_k$.
	
	Let $f \in C_0(0,1]$ be a non-negative function of norm equal to $1$ such that its support is contained in $(0,\delta/2)$. Set $d := f(b)$ which is not zero since $0$ is an accumulation point of $\sigma(b)$. 
	
	We fix $\mu >0$. As in \cite[Lemma 3.2]{HO}, we can assume $c_{ij}d=0$ and hence
	\begin{equation}\label{tlem.eq4}
	\sum_{r=1}^{k} c_{ir} (b-\delta)_+ c_{jr}^* = 
	\begin{cases}
	(a-\epsilon)_+   & i=j, 	\\
	0 & i \neq j.
	\end{cases}
	\end{equation}
	Let $g \in C_0(0,1]$ such that $g|_{[\frac{\mu}{7}, 1]} = 1$ and $g(t)=\sqrt{7t/\mu}$ for $t\in (0,\frac{\mu}{7}]$. Let $h \in C_0(0,1]$ be given by $h(t)=1-\sqrt{1-t}$. Thus
	\begin{equation}\label{tlem.eq5}
	|g(t)^2t-t| < \frac{\mu}{15}, \qquad 1-h(t) = \sqrt{1-t}.
	\end{equation}
	
	Set $\mathfrak{F} = \{(a-\epsilon)_+, (b-\delta)_+, (a-\varepsilon)^{1/2}_+\}\cup\{c_{ij}, c_{ij}(b-\delta)_+ c_{rs}^* \}_{i,j,r,s}$.
	Since we assumed $A$ is $\sigma$-unital, let $(f_n)$ be a sequential increasing approximate unit of $A$ satisfying $f_{n+1}f_n=f_n$ for all $n\geq 1$.  
	Find $n \in \mathbb{N}$ large enough that satisfies
	\begin{equation}\label{tlem.eq11}
	f_n x \approx_{\frac{\mu}{15}} x \approx_{\frac{\mu}{15}} x f_n , \qquad x \in \mathfrak{F}.
	\end{equation}
	Using tracial $\Z$-stability, we find a c.p.c.\ order zero map $\varphi: M_k(\mathbb{C}) \to A$ such that
	\begin{equation} \label{tlem.eq1}
	(f_n (1_{\tilde{A}} - \varphi(1))f_n - \eta)_+ \precsim d, \qquad \|[\varphi,\mathfrak F]\| < \min\left\{ \eta, \frac{\mu}{15}\right\}
	\end{equation}
	where $\eta>0$ is small enough so that
	\begin{equation}
	\label{tlem.eq2}
	\|[g(\varphi), \mathfrak F]\| < {\frac{\mu}{15k^4}},\quad
	\|[\varphi^{1/2}, \mathfrak F]\| < {\frac{\mu}{15k^4}}. 
	\end{equation}
	The existence of such $\eta$ is guaranteed by \cite[Lemma 2.8]{HO}.
	
	Let $m \geq n$ be such that 
	\begin{align}\label{tlem.eq13}
	f_m \varphi(1) &\approx_{\frac{\mu}{15}} \varphi(1).
	\end{align}
	Set
	\begin{align}
	a_1 & := f_m \varphi(1)  (a-\epsilon)_+ f_m, \notag \\
	a_2 & := (a-\epsilon)_+^{1/2} (f_m(1_{\tilde{A}}-\varphi(1))f_m-\eta)_+ (a-\epsilon)_+^{1/2}.
	\end{align}
As $(f_n)$ is an increasing approximate unit and $m\geq n$, it follows from (\ref{tlem.eq11}) that
	\begin{equation}\label{tlem.eq6}
	f_m x \approx_{\frac{\mu}{15}} x, \qquad \text{ for all }x \in \mathfrak{F}.
	\end{equation}
	In particular,
	\begin{align}\label{tlem.eq7}
	a_1 & 
	\overset{\eqref{tlem.eq13}}{\approx}_{\frac{\mu}{15}} \varphi(1)(a-\epsilon)_+ f_m 
	\overset{\eqref{tlem.eq6}}{\approx}_{\frac{\mu}{15}} \varphi(1)(a-\epsilon)_+.
	\end{align}
	Thus
	\begin{align}\label{tlem.eq10}
	a_1 + a_2 &\overset{\eqref{tlem.eq7}}{\approx}_{\frac{2\mu}{15}} \varphi(1) (a-\epsilon)_+ + a_2 \notag \\
	& \approx_{\eta} \varphi(1) (a-\epsilon)_+ + (a-\epsilon)_+^{1/2} f_m (1_{\tilde{A}}-\varphi(1))f_m (a-\epsilon)_+^{1/2}  \notag \\
	& \overset{\eqref{tlem.eq6}}{\approx}_{\frac{2\mu}{15}} \varphi(1) (a-\epsilon)_+ + (a-\epsilon)_+^{1/2} (1_{\tilde{A}}-\varphi(1)) (a-\epsilon)_+^{1/2}  \notag \\
	& \overset{\eqref{tlem.eq1}}{\approx}_{\frac{\mu}{15}} \varphi(1)(a-\epsilon)_+ + (1_{\tilde{A}}-\varphi(1))(a-\epsilon)_+ \notag \\
	& = (a-\epsilon)_+.
	\end{align}
	
	As in the original proof \cite{HO}, we set $g_{ij} := g(\varphi)(e_{ij})$,  $\hat{c}_{ij} := \varphi^{1/2}(1)g_{ij}c_{ij}$ and $\hat{c}:= \sum_{i,j=1}^{k}\hat{c}_{ij}$. 
	Observe that
	\begin{equation}\label{tlem.eq3}
	g_{ij}g_{sr} = 	
	\begin{cases}
	g(\varphi)(1) \cdot g_{ir}  & j=s, \\
	0       & j \neq i.			
	\end{cases}
	\end{equation}
	Then
	\begin{align}\label{tlem.eq8}
	\hat{c} (b-\delta)_+ \hat{c}^* 
	& = \sum_{i,j,r,s=1}^{^k} \varphi^{1/2}(1)g_{ij}c_{ij} (b-\delta)_+ c_{rs}^* g_{sr} \varphi^{1/2}(1) \notag \\
	& \overset{\eqref{tlem.eq2}}{\approx}_{\frac{\mu}{15}} \varphi(1) \sum_{i,j,r,s=1}^{^k} g_{ij}c_{ij} (b-\delta)_+ c_{rs}^* g_{sr}  \notag \\
	& \overset{\eqref{tlem.eq2}}{\approx}_{\frac{\mu}{15}} \varphi(1) \sum_{i,j,r,s=1}^{^k} g_{ij}g_{sr}c_{ij} (b-\delta)_+ c_{rs}^*  \notag \\
	& \overset{\eqref{tlem.eq3}}{=} \varphi(1) \cdot g(\varphi)(1) \sum_{i,j,r=1}^{^k} g_{ir}c_{ij} (b-\delta)_+ c_{rj}^* \notag \\
	& \overset{\eqref{tlem.eq4}}{=} \varphi(1) \cdot g(\varphi)(1) \sum_{i=1}^{^k} g_{ii} \left( \sum_{j=1}^{k}c_{ij} (b-\delta)_+ c_{ij}^*\right) \notag \\
	& \overset{\eqref{tlem.eq4}}{=} \varphi(1) \cdot g(\varphi)(1) \cdot g(\varphi)(1) \cdot (a-\epsilon)_+ \notag \\
	& \overset{\eqref{tlem.eq5}}{\approx}_{\frac{\mu}{15}} \varphi(1) (a-\epsilon)_+ \notag \\
	& \overset{\eqref{tlem.eq7}}{\approx}_{\frac{2\mu}{15}} a_1.
	\end{align}
	On the other hand, 
	\begin{align}\label{tlem.eq14}
	a_2 & = (a-\varepsilon)_+^{1/2}( f_m(1_{\tilde{A}}-\varphi(1))f_m -\eta)_+ (a-\varepsilon)_+^{1/2} \notag \\
	& \overset{\eqref{tlem.eq11}}{\approx}_{\frac{2\mu}{15}} (a-\varepsilon)_+^{1/2} f_n ( f_m(1_{\tilde{A}}-\varphi(1))f_m -\eta)_+ f_n (a-\eta)_+^{1/2} \notag \\
	& \approx_{\eta} (a-\varepsilon)_+^{1/2} f_n f_m(1_{\tilde{A}}-\varphi(1))f_m  f_n (a-\varepsilon)_+^{1/2} \notag \\
	& = (a-\varepsilon)_+^{1/2} f_n (1_{\tilde{A}}-\varphi(1)) f_n (a-\varepsilon)_+^{1/2} \notag \\
	& \approx_{\eta} (a-\varepsilon)_+^{1/2}( f_n(1_{\tilde{A}}-\varphi(1))f_n -\eta)_+ (a-\varepsilon)_+^{1/2}.
	\end{align}
	Since $a$ is a positive contraction, we obtain 
	\begin{align}
	\left(a_2 - \frac{2\mu}{15} - 2 \eta \right)_+ & \overset{\eqref{tlem.eq14}}{\precsim} (a-\varepsilon)_+^{1/2}( f_n(1_{\tilde{A}}-\varphi(1))f_n -\eta)_+ (a-\varepsilon)_+^{1/2} \notag \\
	& \sim ( f_n(1_{\tilde{A}}-\varphi(1))f_n -\eta)_+^{1/2} (a - \varepsilon)_+ ( f_n(1_{\tilde{A}}-\varphi(1))f_n -\eta)_+^{1/2} \notag \\
	& \precsim ( f_n(1_{\tilde{A}}-\varphi(1))f_n -\eta)_+ \notag \\
	& \overset{\eqref{tlem.eq1}}{\precsim} d.
	\end{align}
	Thus there is $s \in A$ such that 
	\begin{equation}\label{tlem.eq9} 
	sds^* \approx_{\frac{\mu}{15}} \left(a_2 - \frac{2\mu}{15}- 2 \eta \right)_+ \approx_{\frac{2\mu}{15}+2\eta} a_2.
	\end{equation}
	As before, we can further assume that $s(b-\delta)_+ =0$ and recall that $c_{ij}d =0$. Then
	\begin{align}
	(\hat{c} + s)((b-\delta)_+ + d)(\hat{c} + s)^* & = \hat{c}(b-\delta)_+ \hat{c}^* + sds^* \notag \\
	& \overset{\eqref{tlem.eq8}}{\approx}_{\frac{\mu}{3}} a_1 + sds^* \notag \\
	& \overset{\eqref{tlem.eq9}}{\approx}_{\frac{\mu}{5} + 2\eta} a_1 + a_2 \notag \\
	& \overset{\eqref{tlem.eq10}}{\approx}_{\frac{\mu}{3}+ \eta} (a-\epsilon)_+.	
	\end{align}
	Since $\mu$ and $\eta$ are arbitrary small, we get
	\begin{equation}
	(a - \epsilon)_+ \precsim (b-\delta)_+ + d \precsim b,
	\end{equation}
	where the last part follows from the construction of $d=f(b)$ where $\mathrm{supp}f \subseteq [0,\delta/2]$. Since $\epsilon$ is arbitrary, we conclude $a \precsim b$.
\end{proof}

We thank the referee for suggesting the following direct proof of the main theorem of this section.

\begin{theorem}[{cf.\ \cite[Theorem 6.4 ]{AGJP}}]\label{thm:strictcomparison}
	Let $A$ be a simple $\C$-algebra. If $A$ is tracially $\mathcal{Z}$-stable, then $\mathrm{Cu}(A)$ is almost unperforated (equivalently, $A$ has strict comparison).
\end{theorem}

\begin{proof}
Let us suppose first that $A$ is $\sigma$-unital.
By \cite[Proposition~2.4]{CE}, we know that $A$ is either stably isomorphic to a unital $\C$-algebra or $A$ is stably projectionless. 
If $A$ is unital, then the result follows from \cite[Theorem 3.3]{HO}, as both tracial $\Z$-stability and the almost unperforation of the Cuntz semigroup are preserved under stable isomorphism. 

If $A$ is stably projectionless, then it is stably finite.
By \cite[Proposition 5.3.16]{APT}, every non-zero element of $\mathrm{Cu}(A)$ is either soft or compact.
By \cite[Theorem 3.5]{BC09}, $\langle a \rangle$ is soft if and only if $\{0\}$ is an accumulation point of $\sigma (a)$ (see Definition \ref{def:soft}). 
Let $\langle a \rangle, \langle b \rangle \in \mathrm{Cu}(A)$ be soft elements such that $(k+1)\langle a \rangle \leq k \langle b \rangle$ for some $k \in \mathbb{N}$.
It follows that $(k+1)\langle a \rangle \leq (k+1) \langle b \rangle$ and, by Lemma \ref{lem:strict.comp}, $\langle a \rangle \leq \langle b \rangle$ in $\mathrm{Cu}(A)$.
Hence the soft part of $\mathrm{Cu}(A)$ is almost unperforated.
By \cite[Proposition 2.8]{Thi20}, $\mathrm{Cu}(A)$ is almost unperforated as well.

For the general case, let $a,b \in (A\otimes \K)_+$ such that there is some $k \in \mathbb{N}$ with $(k+1)\langle a \rangle_A \leq k \langle b \rangle_A$ in $\mathrm{Cu}(A)$.
Let $B$ be a $\sigma$-unital hereditary subalgebra of $A$ such that $B\otimes \K$ contains both $a$ and $b$.\footnote{If we write $a=\sum_{j,\ell\in\N} a_{j,\ell}\otimes e_{j,l}$ and $b=\sum_{j,\ell\in\N} b_{j,\ell}\otimes e_{j,l}$, then the hereditary subalgebra generated by $e=\sum_{j,\ell\in\N} 2^{-(j+\ell)}\big( a_{j,\ell}^*a_{j,\ell}+a_{j,\ell}a_{j,\ell}^*+b_{j,\ell}^*b_{j,\ell}+b_{j,\ell}b_{j,\ell}^*\big)$ would do the trick.}
By \cite[Lemma 2.2]{KR00}, $(k+1)\langle a \rangle_B \leq k \langle b \rangle_B$ in $\mathrm{Cu}(B)$. 
By Proposition \ref{prop:hereditary}, $B$ is tracially $\Z$-stable and hence, by the first part of this proof, $\mathrm{Cu}(B)$ is almost unperforated. 
It follows that $\langle a \rangle_B \leq  \langle b \rangle_B$ in $\mathrm{Cu}(B)$, which clearly yields $\langle a \rangle_A \leq  \langle b \rangle_A$.
This shows that $\mathrm{Cu}(A)$ is almost unperforated.
\end{proof}

\section{$\Z$-stability}\label{main.section}

This is the main section of this note.
We aim to show that tracial $\Z$-stability is equivalent to $\Z$-stability in the separable simple nuclear setting.
We begin with the easy part.

\begin{proposition}\label{prop:Z.implies.tZ}
	Let $A$ be a simple $\Z$-stable $\C$-algebra. 
	Then $A$ is tracially $\Z$-stable. 
\end{proposition}
\begin{proof}
	By assumption we have $A\cong A\otimes\Z$, so by Proposition \ref{prop:stabilisation} it suffices to know that $\Z$ is itself tracially $\Z$-stable.
	But this is well-known, see for example \cite[Proposition 2.2]{HO}. 
\end{proof}

The notion of property (SI), introduced by Matui and Sato in \cite{MS12}, has been fundamental in many recent developments in the structure and classification of simple nuclear $\C$-algebras. 
It was originally introduced for simple unital $\C$-algebras in \cite[Definition 4.1]{MS12} and has been recently revised by the third named author in \cite{Szabo19}, in which a general framework was developed to cover the class of all simple separable nuclear $\C$-algebras.
Let us record some of the ingredients needed in this note.

\begin{definition}[{\cite[Definition 2.5]{Szabo19}}]
	Let $A$ be a separable simple $\C$-algebra. 
	\begin{enumerate}[(i)]
		\item A positive contraction $f \in F_\omega(A)$ is called \emph{tracially supported at 1}, if one of the following is true: $A$ is traceless and $\|fa\|= \|a\|$ for all $a \in A_+$; or $T^+(A)\neq\{0,\infty\}$ and for all non-zero positive $a \in \mathrm{Ped}(A)$ there exists a constant $\kappa_{f,a}>0$ such that
		\begin{equation}
		\inf_{m \geq 1} \tau_a (f^m) \geq \kappa_{f,a} \tau(a) \notag
		\end{equation} 
		for all $\tau \in T_\omega^+ (A)$ with $\tau|_A$ non-trivial.
		
		\item A positive element $e \in F_\omega(A)$ is called \emph{tracially null} if $\tau_a(e) = 0$ for all  non-zero positive $a \in \mathrm{Ped}(A)$ and $\tau \in T_\omega^+(A)$ with $\tau(a) < \infty$.
	\end{enumerate}
	Either one of the conditions above holds for all non-zero positive $a\in \mathrm{Ped}(A)$ if it holds for just one such element.
\end{definition}

\begin{definition}[{\cite[Definition 2.7]{Szabo19}}]
	Let $A$ be a separable simple $\C$-algebra. It is said that $A$ has \emph{property (SI)} if whenever $e,f \in F_\omega(A)$ are positive contractions with $f$ tracially supported at 1 and $e$ tracially null, there exists a contraction $s\in F_\omega(A)$ with
	\begin{equation}
	fs = s \quad \text{and} \quad s^*s =e. 
	\end{equation}
\end{definition}

Importantly, it follows from \cite[Corollary 3.10]{Szabo19} that non-elementary separable simple nuclear $\C$-algebras with strict comparison have property (SI). 
For completeness we include a proof of the following folklore result which is well-known to the experts.
The underlying argument has appeared in the literature several times before. 

\begin{proposition}[Matui--Sato]\label{prop:Matui-Sato}
	Let $A$ be a simple, separable and nuclear $\C$-algebra with strict comparison and $T^+(A) \neq \{0,\infty\}$.
	Then $A$ is uniformly McDuff if and only if $A$ is $\Z$-stable. 
\end{proposition}

\begin{proof}
The ``if'' part is clear by Remark \ref{rem:Z-is-McDuff}, so we proceed to prove the ``only if'' part.
	
	It follows from \cite[Corollary 3.10]{Szabo19} that $A$ has property (SI).
	For each $n \in \mathbb{N}$, there is a tracially large order zero map $\varphi: M_n(\mathbb{C}) \to F_\omega(A)$. Set $e:= 1_{F_\omega(A)} - \varphi(1)\in F_\omega(A)$ and $f:=\varphi(e_{11})\in F_\omega(A)$. 
	The fact that $\varphi$ is tracially large means precisely that $e$ is tracially null.
	Let us check that $f$ is tracially supported at 1.
	
	Since the tracially null elements form an ideal and $\varphi(1)$ agrees with the unit of $F_\omega(A)$ modulo this ideal, we also get that $1_{F_\omega(A)}-\varphi(1)^m$ is tracially null for any $m\geq 1$. This shows that $\left(1_{F_\omega(A)}-\varphi(1)^m\right)f$ is tracially null and
	\begin{equation}
		\tau_a(f) = \tau_a(\varphi(1)^m f), \qquad m \in \mathbb{N}.
	\end{equation}
	Recall that by the structure theorem of order zero maps, we have $\varphi(e_{11})^m = \varphi(1)^{m-1} \varphi(e_{11})$ (see Section \ref{section:orderzero}). Thus
	\begin{equation}
	\tau_a (f^m) = \tau_a(\varphi(e_{11})^m)=\tau_a(\varphi(1)^{m-1}\varphi(e_{11}))=\tau_a (f) 
	\end{equation}
	for all non-zero positive $a \in \mathrm{Ped}(A)$ and $\tau \in T^+_\omega(A)$ with $\tau(a) < \infty$.
	Observe that, by uniqueness of the tracial state $\mathrm{tr}$ on $M_n(\mathbb C)$, $\tau_a\circ\varphi$ must be a multiple of $\mathrm{tr}$. Then
	\begin{align}
		\tau_a (f^m) = \tau_a (f) = \mathrm{tr}(e_{11}) \tau_a (\varphi(1))  = \frac{1}{n}\tau(a)
	\end{align} 
	for all non-zero positive $a \in \mathrm{Ped}(A)$ and $\tau \in T^+_\omega(A)$ with $\tau(a) < \infty$.
	It follows that $f$ is tracially supported at $1$.
	
	By property (SI), there exists $s \in F_\omega(A)$ such that $fs=s$ and $s^*s = e$. By \cite[Proposition 5.1(iii)]{RW10}, there is a unital $^*$-homomorphism from the dimension drop algebra\footnote{Given $n,m \in \mathbb{N}$, the dimension drop algebra $Z_{n,m}$ is defined as $
	Z_{n,m} := \{ f \in C([0,1], M_n(\mathbb{C})\otimes M_m(\mathbb{C}))\mid f(0) \in M_n(\mathbb{C}) \otimes 1  , f(1) \in 1\otimes M_m(\mathbb{C})\}.
$} $Z_{n,n+1}$ into $F_\omega(A)$ for all $n \in \mathbb{N}$. By \cite[Proposition 5.1]{Na13}, we conclude that $A$ is $\Z$-stable.
\end{proof}

Next we will show that tracial $\Z$-stability implies the uniform McDuff property.
Let us prove a preliminary lemma first.

\begin{lemma}[{cf.\ \cite[Lemma 4.3]{HO}}]\label{lem:small.c}
	Let $A$ be a simple non-elementary $\C$-algebra with $T^+(A)\neq\{0,\infty\}$.
	Suppose that $K\subset T^+(A)\setminus\{0,\infty\}$ is a compact subset.
	Then for any $n \in \mathbb{N}$, there exists a positive element $c_n\in A$ of norm one such that $d_\tau(c_n) \leq \frac{1}{n}$ for all $\tau\in K$.
\end{lemma}

\begin{proof}
	Let $b$ be a non-zero positive element in $\mathrm{Ped}(A)$.
	Then by compactness, we have $\sup_{\sigma\in K} d_\sigma(b)<\infty$, so let us choose a natural number $k\in\N$ greater than this constant.
	In particular, each trace $\tau\in K$ restricts to a positive tracial functional of norm at most $k$ on the hereditary subalgebra $\overline{bAb}$.
	Using \cite[Proposition~4.10]{KR00} and \cite[Corollary 4.1]{WZ09}, there exists a c.p.c.\ order zero map $\psi: M_{nk}(\mathbb{C}) \to \overline{bAb}$ such that $c_n:= \psi(e_{11})$ is a positive contraction of norm one in $A$. 
Given $\tau\in K$, it follows that
\begin{equation} 
nk\cdot d_\tau(\psi(e_{11}))=d_\tau(\sum_{i=1}^{nk}\psi(e_{ii})) = d_\tau(b)\leq k.
\end{equation}
Therefore $c_n$ satisfies the required property.
\end{proof}

\begin{proposition}\label{pro:tZ.tlarge.maps}
	Let $A$ be a separable simple $\C$-algebra with $T^+(A) \neq \{0,\infty\}$.
	If $A$ is tracially $\Z$-stable, then $A$ is uniformly McDuff. 
\end{proposition}

\begin{proof}
Pick any non-zero positive element $b\in\mathrm{Ped}(A)$, and define the subset $K\subset T^+(A)\setminus\{0,\infty\}$ as those traces that normalize $b$.
Then $K$ is clearly compact and every non-trivial trace on $A$ is a constant multiple of a trace in $K$.
	Let $c_k \in A$ be the positive elements of norm one given from Lemma \ref{lem:small.c} for $k \in \mathbb{N}$ and let $c \in B_\omega$ be the induced element. 
	We claim that $\tau(c) = 0$ for any $\tau \in T^+_\omega(A)$ with $\tau(b)<\infty$.
	Indeed, clearly $\tau(c)=0$ whenever $\tau$ is induced from a sequence $\tau_n\in K$.
	However, if $\tau(b)<\infty$, then by \cite[Lemma~2.10, Remark~2.11]{Szabo19}, $\tau$ is already a constant multiple of some generalised limit trace induced from a sequence in $K$, so the claim follows.
	
	Let $n\in\N$.
	By Proposition \ref{prop:tZ.ultrapower}, there are a c.p.c.\ order zero map $\varphi: M_n(\mathbb{C}) \to A_\omega \cap A'$ and a contraction $x\in A_\omega$ such that 
	\begin{equation}
		cx=x \quad\text{and}\quad x^*x-(1_{\tilde A_\omega}-\varphi(1_{M_n})) \in \tilde{A}_\omega\cap A^\perp. \label{prop:tlarge.maps.eq}
	\end{equation}
	 
	Let us show that this map induces a tracially large map into $F_\omega(A)$.
	Indeed we have for any $\tau \in T^+_\omega(A)$ with $\tau(b)<\infty$ that
	\[
	\begin{array}{ccl}
		\tau(b^{1/2}(1_{\tilde A_\omega} - \varphi(1_{M_n}))b^{1/2}) 
		&\overset{\eqref{prop:tlarge.maps.eq}}{=}&  \tau( b^{1/2} x^* c x b^{1/2}) \\
		&=& \tau(c^{1/2}xbx^*c^{1/2}) \\
		& \leq & \|x\|^2\|b\|\tau(c) = 0.
	\end{array}
	\]	
	This shows that $1-\varphi(1_{M_n})$ vanishes under the trace $\tau_b$.
	As explained in the footnote at Definition \ref{def:traciallylargemap},  \cite[Proposition 2.4]{Szabo19} yields that $1-\varphi(1_{M_n})$ vanishes under any trace $\tau_a$ with $a\in \mathrm{Ped}(A)$ and $\tau\in T^+_\omega(A)$ satisfying $\tau(a) < \infty$.
	Hence the induced map $\bar \varphi: M_n(\mathbb{C}) \to F_\omega(A)$ is tracially large and $A$ is uniformly McDuff.
\end{proof}

\begin{theorem}\label{main thm}
	Let $A$ be a separable simple nuclear $\C$-algebra.
	If $A$ is tracially $\Z$-stable, then $A$ is $\Z$-stable.
\end{theorem}

\begin{proof}
By \cite[Proposition~2.4]{CE}, we know that $A$ is either stably projectionless or $A$ is stably isomorphic to a unital $\C$-algebra. 
If $A$ is stably isomorphic to a unital $\C$-algebra, say $B$, it follows by Corollary \ref{stable iso} that $B$ is tracially $\Z$-stable. By \cite[Theorem 4.1]{HO}, it follows that $B$ is $\Z$-stable. 
Since $\Z$-stability is preserved under stable isomorphism, it follows that $A$ is $\Z$-stable.
On the other hand, if $A$ is stably projectionless, then $A$ is stably finite.
Recall that $A$ has strict comparison by Theorem \ref{thm:strictcomparison}.
Thus, Propositions \ref{pro:tZ.tlarge.maps} and \ref{prop:Matui-Sato} yield that $A$ is $\Z$-stable.
\end{proof}

Matui introduced the notion of \emph{almost finiteness} for étale groupoids with a compact and totally disconnected unit space (see \cite[Definition~6.2]{MR2876963} and we also refer the reader to \cite{MR584266} for the theory of étale groupoids). Recently, the notion of strongly almost finiteness was introduced in \cite[Definition~3.12]{ABBL} for étale groupoids $G$ with a (not necessarily compact) totally disconnected unit space. More precisely, $G$ is \emph{strongly almost finite} if the restriction $G|_K$ is almost finite in the sense of Matui for all compact open subsets $K$ of the unit space $G^{(0)}$. So the reduced groupoid $\C$-algebra $C_r^*(G)$ may \emph{not} be unital in general. Moreover, it is known that when the groupoid is also minimal and has a compact unit space, strong almost finiteness agrees with Matui's almost finiteness (see \cite[Proposition 3.6]{ABBL}). It is also worth noting that (strong) almost finiteness does not imply amenability for groupoids nor exactness for groupoid $\C$-algebras (see \cite[Theorem 6]{Elek18} and \cite[Remark~2.10]{MR4133645}). 

\begin{corollary}
Let $G$ be a locally compact Hausdorff minimal $\sigma$-compact étale groupoid with totally disconnected unit space $G^{(0)}$ without isolated points. If $G$ is strongly almost finite, then the reduced $\C$-algebra $C_r^*(G)$ is a $\sigma$-unital simple and tracially $\Z$-stable $\C$-algebra with real rank zero and stable rank one.

If $G$ is also amenable and second-countable, then $C_r^*(G)$ is classifiable by its Elliott invariant and has decomposition rank at most one.
\end{corollary}
\begin{proof}
Fix a compact open subset $K\subseteq G^{(0)}$. Since $G$ is minimal, it follows that $G$ and the restriction $G|_K$ are Morita equivalent. Hence, $C_r^*(G)$ and $C_r^*(G|_K)$ are stably isomorphic by \cite[Theorem 2.1]{MR3601549}. On the other hand, $G|_K$ is a minimal, almost finite,  $\sigma$-compact étale groupoid with compact totally disconnected unit space $K$. As $K$ is clopen in $G^{(0)}$, $K$ has no isolated points as well. By \cite[Corollary~9.11]{MW20}, we deduce that $C_r^*(G|_K)$ is a unital simple and tracially $\Z$-stable $\C$-algebra. Moreover, $C_r^*(G|_K)$ has real rank zero and stable rank one (see \cite{MR4165470, MR4133645}). Hence, we conclude that $C_r^*(G)$ satisfies the desired properties by Corollary~\ref{stable iso}.

If we also assume amenability and second-countability for $G$, then $C_r^*(G)$ is also a nuclear separable $\Z$-stable $\C$-algebra in the UCT class by Theorem~\ref{main thm} and \cite{MR1703305}.
Notice that by \cite[Lemma 3.9]{MR4165470}, we have $T^+(C_r^*(G))\neq\{0,\infty\}$.
Finally, \cite[Theorem~7.2 and Remark~7.3]{CE} together imply that $C_r^*(G)$ has decomposition rank at most one as desired.
\end{proof}


\medskip 



\end{document}